\documentclass{jTmodifie}

\usepackage{amssymb,amsmath,latexsym,amsthm} 
\usepackage[T1]{fontenc}
\usepackage[francais]{babel} 
\newtheorem{theoreme}{Theorem}[section]
\newtheorem{lemme}{Lemma}[section]
\newtheorem{proposition}{Proposition}[section]
\newtheorem{corollaire}{Corollary}[section]
\newtheorem{conjecture}{Conjecture}[section]
\theoremstyle{definition}

\newtheorem*{remarque}{Remark}

\numberwithin{equation}{section}

\def\Q{\mathbb{ Q}}

\def\N{\mathbb{N}} 
\def\Z{\mathbb {Z}}

\def\Norm{{\rm N}}

\def\rmh{\mathrm{h}}

\def\BigOh{O}

\newcounter{compteurkappa} 

\def\Newcst#1{
\refstepcounter{compteurkappa}
\kappa_{ 
\arabic{compteurkappa}}
\label{#1}
}

\def\cst#1{\kappa_{\ref{#1}}}
\def\boxit#1#2{\setbox1=\hbox{\kern#1{#2}\kern#1}%
\dimen1=\ht1 \advance\dimen1 by #1 \dimen2=\dp1 \advance\dimen2 by #1
\setbox1=\hbox{\vrule height\dimen1 depth\dimen2\box1\vrule}%
\setbox1=\vbox{\hrule\box1\hrule}%
\advance\dimen1 by .4pt \ht1=\dimen1
\advance\dimen2 by .4pt \dp1=\dimen2 \box1\relax} 

\begin{document}
 
\def\refname{\centerline{Bibliographie}}
\title[A family of Thue equations]{A family of Thue equations involving powers of units of the simplest cubic fields}

\author{\sc Claude Levesque} 
\address{Claude Levesque\\
D\'ep. de math\'{e}matiques et de statistique\\
Universit\'{e} Laval, Qu\'{e}bec (Qu\'{e}bec)\\
CANADA G1V 0A6}
\email{cl@mat.ulaval.ca}

\author{\sc Michel Waldschmidt}
\address{Michel Waldschmidt\\
UPMC Univ Paris 06, UMR 7586-IMJ\\
F--75005 Paris France 
}

\email{michel.waldschmidt@imj-prg.fr}
\urladdr{http://www.imj-prg.fr/~michel.waldschmidt/}

\maketitle

\begin{resume}
E. Thomas fut l'un des premiers \`a r\'esoudre une famille infinie
 d'\'equations de Thue, lorsqu'il a consid\'er\'e les formes 
$F_n(X, Y )= X^3 -(n-1)X^2Y -(n+2)XY^2 -Y^3$
et la famille d'\'equations $F_n(X, Y )=\pm 1$, $n\in \N$. 
Cette famille est associ\'ee \`a la famille des corps cubiques les plus simples 
$\Q(\lambda)$ de D.~Shanks,
$\lambda$ \'etant une racine de $F_n(X,1)$. Nous introduisons dans cette famille
un second param\`etre en rempla\c cant les racines du polyn\^ome minimal
$F_n(X, 1) $ de $\lambda$ par les puissances $a$-i\`emes des racines 
et nous r\'esolvons de fa\c con effective la famille d'\'equations de Thue 
que nous obtenons et qui d\'epend maintenant des deux param\`etres $n$ et $a$.

\end{resume}

\begin{abstr}{\sl A family of Thue equations involving 
powers of units of the simplest cubic fields}
 
 E. Thomas was one of the first to solve an infinite family 
of Thue equations, when he considered the forms 
$F_n(X, Y )= X^3 -(n-1)X^2Y -(n+2)XY^2 -Y^3
$
and the family of equations $F_n(X, Y )=\pm 1$, $n\in \N$. 
This family is associated to the family of the simplest cubic fields $\Q(\lambda)$ of D.~Shanks,
$\lambda$ being a root of $F_n(X,1)$. We introduce in this family a second parameter by replacing the roots of the minimal polynomial $F_n(X, 1) $ of $\lambda$ by the $a$-th powers of the roots and we effectively solve the family of Thue equations that we obtain and which depends now on the two parameters $n$ and $a$. 
 \end{abstr}
 
  \bigskip
 \noindent
{\tt Key words:} Simplest cubic fields, family of Thue equations, diophantine equations, linear forms of logarithms.

\bigskip
 \noindent
{\tt AMS Classification: } Primary: 11D61, $\quad$ Secondary: 11D41, 11D59.

\section{Introduction}

In 1990, E.~Thomas \cite{T} considered the family $(F_n)_{n\ge 0}$ of binary cubic irreducible forms $F_n(X, Y )\in\Z[X,Y]$ associated to the simplest cubic cyclic fields of D.~Shanks, namely 
$$
F_n(X, Y )= X^3 -(n-1)X^2Y -(n+2)XY^2 -Y^3 .
$$
In particular, Thomas proved in some effective way that the set of $(n,x,y)\in\Z^3$ with 
$$
n\ge 0, \quad 
\max\{|x|, |y|\} \ge 2 
\quad\hbox{and}\quad F_n(x,y)=\pm 1
$$
 is finite. 
In his paper, he completely solved the case $n\ge 1.365\cdot 10^7$. In 1993, M.~Mignotte \cite{M} completed the work of Thomas by solving the problem for each $n$ (see \S\ref{Section:NumericalCalcutation}). 
In 1996, M.~Mignotte, A.~Peth\H{o} and F.~Lemmermeyer \cite{MPL} gave, for $m\ge 1$ and $n\ge 1650$, an explicit upper bound of $|y|$ when $x,y$ are rational integers verifying
 $$
 0< |F_n(x,y)|\leq m.
 $$ 
Theorem 1 of \cite{MPL} implies, for $n\ge 2$, 
 $$
 \log |y|\le \Newcst{MPL} (\log n)(\log n+\log m)
 $$
 with an absolute constant $\cst{MPL}>0$. 

\smallskip

Following \cite{T} (compare with \cite{LPV}), let us note 
$$
\lambda_0, \quad \lambda_1=-\frac{1}{\lambda_0+1}
 \quad\hbox{and}\quad
 \lambda_2=-\frac{\lambda_0+1}{\lambda_0}\cdotp
 $$
the real roots of the polynomial $f_n(X)=F_n(X,1)\in\Z[X]$ with the convention $$
\lambda_0>0>\lambda_1>-1>\lambda_2.
$$
 These numbers depend on $n$, but no index $n$ appears to avoid a heavy notation.
 
 \smallskip

When $a$ is a nonzero rational integer, the binary cubic form 
 $$
 F_{n,a}(X,Y)=(X-\lambda_0^a Y)(X-\lambda_1^a Y)(X-\lambda_2^a Y)
$$
is irreducible in $\Z[X,Y]$, the minimal (irreducible) polynomial of $\lambda_0^a$ being $F_{n,a}(X,1)$. 

\smallskip

The purpose of this paper is to prove the following theorem.

\begin{theoreme}\label{Theoreme:principal} Let $m\in \N=\{1,2,\dots\}$.
There exist some absolute effectively computable constants 
$\Newcst{theoreme:m=1}$, $\Newcst{theoreme:2}$, $\Newcst{kappa:agrand}$ and $\Newcst{kappa:xypetits}$ such that if there exists $(n,a,m,x,y)\in\Z^5$ with $a\neq 0$ verifying
$$
0<|F_{n,a}(x,y)|\leq m,
$$
then the following properties hold true: 

{\rm (i)} When $m= 1$ and 
$\max\{|x|, |y|\}\geq 2$, 
we have
$$
\max\{|n|, |a|, |x|, |y|\}\leq \cst{theoreme:m=1}.
$$

{\rm (ii)} 
We have 
$$
\log \max\{|x|, |y|\}\leq \cst{theoreme:2} \mu
$$
with 
$$
\mu=\left\{
\begin{array}{ll}
(\log m+ |a| \log |n|)(\log |n|)^2 \log \log |n|
&\mbox{for $|n|\ge 3$,}
\\[2mm]
\log m+ |a| &\mbox{for $n=0,\pm1,\pm 2$.}
\end{array}
 \right.
$$ 

{\rm (iii)} When $n\ge 0$, $a\ge 1$ and 
$|y|\ge 2\root 3\of m$, 
we have
$$
a \leq \cst{kappa:agrand} \mu'
$$
with 
$$
 \mu'=
 \begin{cases}
 (\log m+\log n) (\log n) \log\log n
 &\mbox{for $n\ge 3$,}
\\[2mm]
 1+ \log m &\mbox{for $n=0, 1, 2$.}
\end{cases}
 $$

{\rm (iv)} When $xy\not=0$, $n\ge 0$ and $a\ge 1$,
 we have
$$
a \leq \cst{kappa:xypetits} 
\max\left \{
1,\; (1+\log |x|)\log\log(n+3),\; \log|y|,\; \frac{\log m}{\log(n+2)}
\right\}.
 $$

\end{theoreme}

Part (i) of Theorem $\ref{Theoreme:principal}$ extends the result of Thomas to the family of Thue equations 
 $ F_{n,a}(X,Y)=\pm 1$. Part (ii) generalizes the result of Mignotte, Peth\H{o} and Lemmermeyer to the family of Thue inequations $ |F_{n,a}(X,Y)|\le m$ (though in the case $a=1$ our bound for $|y|$ in (ii) is not as sharp as the one of \cite{MPL}). Moreover, the part (iii) of our theorem takes care of the family of all Thue inequations $ |F_{n,a}(X,Y)| \leq m$ in the case where $|a|$ is sufficiently large with respect to $n$. 

\smallskip
Since 
$$
F_{n,1}(X,Y)=F_n (X,Y)\quad \mbox{and} \quad F_{n,-1}(X,Y)=-F_n (Y,X),
$$
when $a=\pm 1$, the part (i) of Theorem $\ref{Theoreme:principal}$ boils down to the theorem of \cite{T} quoted 
above.

\smallskip
The result of Thomas \cite{T} shows that there is only a finite set of triples $(n,x,y)$ 
such that $\max\{|x|, |y|\} \ge 2$ and $x-\lambda_0 y$ is a unit. The part (i) of Theorem $\ref{Theoreme:principal}$ 
shows that there is only a finite set of quadruples $(n,a,x,y)$ 
with $\max\{|x|, |y|\} \ge 2$ and $ a\not=0 $ such that $x-\lambda_0^a y$ is a unit and our method provides an algorithm for exhibiting them all. 
 The assumption $\max\{|x|, |y|\}\geq 2$ in (i) cannot be omitted, since $F_{n,1}(1,-1)=1$ for all $n\ge 0$.
 
 \smallskip

 When $a=\pm 1$, the part (ii) of Theorem $\ref{Theoreme:principal}$ is a consequence of the Theorem 1 of \cite{MPL} quoted above.

\smallskip
The proof of the part (i) of Theorem 1, in the case
$$ 
2\le |a| \le \frac{{|n|}}{(\log |n|)^4},
$$
will be done in parallel with the proof of the part (ii), while 
in the case where 
 $|a|$ is sufficiently large with respect to $|n|$, namely
$$
\mbox{$|a|> \cst{kappa:agrand} (\log |n|)^2 \log\log |n|$ 
for $|n|\ge 3 \quad$ and $\quad |a|> \cst{kappa:agrand}$ for $n=0,\pm 1,\pm 2$},
 $$
the parts (i) and (ii) of Theorem 1 are consequences
 of the part (iii).

\smallskip

The proof of the part (iv) of Theorem $\ref{Theoreme:principal}$ also involves a Diophantine 
argument which will be given in \S$\ref{Section:Demonstration(iv)}$. 
\smallskip

 Under the hypotheses of Theorem $\ref{Theoreme:principal}$, assuming $xy\not=0$ and $|a|\ge 2$, we conjecture the upper bound
 $$
 \max\{\log |n|,|a|,\log|x|,\log|y|\}\le \Newcst{kappa:conjecture}(1+\log m),
 $$ 
though for $m>1$ we cannot give an upper bound for $|n|$. Since the rank of the units of $\Q(\lambda_0)$ is $2$, one may expect a more general result as follows:

\begin{conjecture}
For $s,t$ and $n$ in $\Z$, define
$$
F_{n,s,t}(X,Y)= (X-\lambda_0^{s}\lambda_1^{t}Y)
(X-\lambda_1^{s}\lambda_2^{t}Y)
(X-\lambda_2^{s}\lambda_0^{t}Y).
$$
There exists a positive absolute constant $\Newcst{kappa:conjectureRang2}$ with the following property: 
If $n,s,t,x,y,m$ are integers satisfying 
$$
\max\{|x|, |y|\}\ge 2, \quad (s,t)\not=(0,0)
\quad\hbox{and}\quad 0<|F_{n,s,t}(x,y)|\le m,
$$
then 
$$
\max\{\log |n|, |s|, |t|, \log |x|, \log |y|\}\le \cst{kappa:conjectureRang2} (1+\log m).
$$
\end{conjecture}

\smallskip
The parts (ii) and (iii) of Theorem $\ref{Theoreme:principal}$ involve a family of Thue inequations $|F_{n,a}(x,y)|\le m$ with $m\ge 1$, but it suffices to consider for each
\linebreak
$m_0\in \{1,2, \dots, m\}$ the family of Thue equations 
$$
|F_{n,a}(x,y)|=m_0,
$$
where the unknowns $(n,a,x,y)$ take their values in $\Z^4$. 
 The constants $\kappa$ which will appear (often implicitly, in the $\BigOh$'s) 
 are positive numbers (absolute constants), the existence of which is postulated. 

\smallskip

 To prove parts (i) and (ii) of 
 Theorem $\ref{Theoreme:principal}$, 
 there is no restriction in assuming $n\ge 0$, thanks to the equality
$$
F_{-n-1}(X,Y)=F_n(-Y,-X)
$$
which implies
$$
F_{-n-1,a}(X,Y)=F_{n,a}(-Y,-X).
$$
In the same vein, the relation 
$$
F_{n,-|a|}(X,Y)=-F_{n,|a|}(Y,X)
$$
 allows to suppose $a\ge 1$. As a consequence of these two relations, in the part (iii), we may replace the assumptions 
 $$
 \hbox{$n\ge 0$, $a\ge 1$ and 
$|y|\ge 2\root 3\of m$}
$$
by
 $$ 
\min\{|x|,\; |y|\}\ge 2\root 3\of m,
$$
provided that we replace $n$ and $a$ by $|n|$ and $|a|$ respectively in the conclusion. 

Moreover, because of the relation 
$$
F_{n,a}(X,Y)=-F_{n,a}(-X,-Y),
$$
to prove each part of Theorem $\ref{Theoreme:principal}$, we may assume $y\ge 0$. 
Since the parts (i) and (ii) of Theorem $\ref{Theoreme:principal}$ are not new for $a=1$, we may assume $a\ge 2$ (which is no restriction for the parts (iii) and (iv)). 

\smallskip
It happens that this paper is the fourth one in which we use the effective method arising from Baker's work on linear forms of logarithms for families of Thue equations obtained via a twist of a given equation by units. The first paper \cite{LW-JAustralMS} was dealing with non totally real cubic fields;
 the second one \cite{LW-Balu} was dealing with Thue equations attached to a number field having at most one real embedding. In the third paper \cite{LW}, for each (irreducible) binary form attached to an algebraic number field, which is not a totally real cubic field, we exhibited an infinite family of equations twisted by units for which Baker's method provides effective bounds for the solutions. In this paper, we deal with equations related to infinite families of cyclic cubic fields. In a forthcoming paper \cite{LWrang1}, we go one step further by considering twists by a power of a totally real unit. 
 
 \smallskip
 
\vskip 0.5cm
\section{Related results} 

Let us state other results involving the families of Thomas. 
Theorem 3 of M.~Mignotte, A.~Peth\H{o} and F.~Lemmermeyer \cite{MPL}, 
published in 1996, gives a complete list of the triples $(x,y,n)\in\Z^3$ with $n\geq 0$ such that 
$$
0<|F_n(x,y)|\le 2n+1.
 $$
In particular, when $m$ is a given positive integer, there exists an integer $n_0$ depending upon $m$ such that the inequality $|F_n(x,y)|\leq m$ with $n\geq 0$ and $|y|>\root 3 \of m$ implies $n\leq n_0$. 
Note that for $0<|t|\leq \root 3 \of m$, $(-t,t)$ and $(t,-t)$ are solutions. Therefore, the condition $|y|>\root 3 \of m$ cannot be omitted. 

\smallskip
 Also in 1996, for the family of Thue inequations 
 $$
 0<|F_n(x,y)|\leq m,
 $$ 
 J.H.~Chen \cite{C} has given a bound for $n$ by using Pad\'e's approximations. This bound was highly improved in 1999 by G.~Lettl, A.~Peth\H{o} and P.~Voutier \cite{LPV}. More recently, I.~Wakabayashi \cite{W}, using again the approximants of Pad\'e, extended these results to the families of forms, depending upon two parameters, 
$$
sX^3 -tX^2Y -(t+3s)XY^2 -sY^3, 
$$ 
which includes the family of Thomas for $s=1$ (with $t=n-1$). Wakabayashi considered also the family of quartic forms
 $$
 sX^4-tX^3Y-6sX^2Y^2+tXY^3+sY^4
 $$
and the family of sextic forms
 $$ 
 sX^6 -2tX^5Y -(5t+15s)X^4Y^2
-20sX^3Y^3 +5tX^2Y^4 +(2t+6s)XY^5 +sY^6.
 $$

When the following polynomials are irreducible for $s,t\in\Z$, 
the fields $\Q(\omega)$ generated by a root $\omega$ of respectively 
$$\left\{
\begin{array}{l}
sX^3-tX^2-(t+3s)X-s,
\\[2mm]
sX^4-tX^3-6sX^2+tX+s,
\\[2mm]
sX^6 -2tX^5 -(5t+15s)X^4-20sX^3 +5tX^2 +(2t+6s)X +s,
\end{array} \right.
$$
are cyclic over $\Q$ of degree 3, 4 and 6 respectively and are called {\it simplest fields} by many authors. In each of the three cases, the roots of the polynomials can be described via homographies of $PSL_2(\Z)$ of degree 3, 4 and 6 respectively. See \cite{LPV,W}. 
 \vskip 0.5cm
 
\section{A side remark}

The parts (ii) and (iii) of Theorem $\ref{Theoreme:principal}$ 
 may be viewed as writing lower bounds for $|F_{n,a}(x,y)|$. In some opposite direction, we can exhibit quadruples 
\linebreak
$(n,a,x,y)\in\Z^4$ such that $|F_{n,a}(x,y)|$ takes relatively small values.
 
\begin{lemme}\label{Lemme:Fnapetit}
Let $n\ge 0$ and $a\ge 1$ be integers. There exist infinitely many $(x,y)\in\Z^2$ with $y>0$ 
such that 
$$
|F_{n,a}(x,y)|\le y (n+4)^a. 
$$
\end{lemme}

The proof of Lemma $\ref{Lemme:Fnapetit}$ will use the following bounds for 
 the three roots of $f_n(X) = F_{n,1} (X,1)$ for $n\geq 2$; the bounds are given in  \cite{LPV} and are easily validated with  {\sl MAPLE}:

 the bounds are given in  \cite{LPV}
\begin{equation}\label{equation:boundslambda}
\left\{
\begin{array}{rccllclll}\displaystyle
n+\displaystyle\frac{1}{n} &\!\!\!< \!\!\!&n+\displaystyle\frac{2}{n+1} &\!\!\!< \!\!\!& \lambda_0&\!\!\! < \!\!\!& \displaystyle n+\frac{2}{n},
\\[5mm]
\displaystyle 
-\frac{1}{n+1}&\!\!\!<\!\!\!& -\displaystyle\frac{1}{n+1+\frac{1}{n}} &\!\!\!<\!\!\!& \lambda_1 &\!\!\!< \!\!\!&\displaystyle -\frac{1}{n+1+\frac{2}{n}} &\!\!\!\le\!\!\!& -\displaystyle\frac{1}{n+2}, 
\\[5mm]\displaystyle
-1-\displaystyle\frac{1}{n}&\!\!\!<\!\!\!& -1-\displaystyle\frac{n}{n^2+1} &\!\!\!<\!\!\!& \lambda_2 &\!\!\!< \!\!\!&\displaystyle -1 -\frac{n}{n^2+2} 
&\!\!\!\le\!\!\!& -1-\displaystyle\frac{1}{n+1}\cdot
\\[2mm]
\end{array}\right.
\end{equation}
For $n=1, \; f_1(X)=X^3-3X-1=(X-\lambda_0)(X-\lambda_1)(X-\lambda_2) $ and we have
$$
1.8793<\lambda_0<1.8794, \quad -0.3473<\lambda_1<-0.3472, \quad -1.5321<\lambda_2<-1.532.
$$

\begin{proof}[\indent Proof of Lemma $\ref{Lemme:Fnapetit}$.] 
Let $\epsilon$ be a positive real number. We will prove a refined version of Lemma $\ref{Lemme:Fnapetit}$ according to which 
there exist infinitely many $(x,y)\in\Z^2$ with $y>0$ such that 
$$ 
|F_{n,a}(x,y)| 
\le \frac{1}{2}y (1+\epsilon)
|\lambda_2^a-\lambda_1^a|\cdotp
|\lambda_0^a-\lambda_2^a|
\cdotp 
$$ 
 Lemma $\ref{Lemme:Fnapetit}$ will follow from the upper bound
$$ 
|\lambda_2^a-\lambda_1^a|\cdotp
|\lambda_0^a-\lambda_2^a|
< 2 (n+4)^a
$$
by selecting $\epsilon$ sufficiently small so that
$$ 
(1+\epsilon)|\lambda_2^a-\lambda_1^a|\cdotp
|\lambda_0^a-\lambda_2^a|
 \le 2 (n+4)^a. 
$$ 
We may assume 
$\epsilon<3$ so that 
$$
\left(1+\frac{\epsilon}{3}\right)^2<1+\epsilon.
$$
Let $Y$ be a sufficiently large integer so that
$$
\frac{1}{2Y^2}\le \frac{\epsilon}{3} \min\left\{
|\lambda_2^a-\lambda_1^a|\; , \; 
|\lambda_0^a-\lambda_2^a|\right\}.
$$
Thanks to the theory of Diophantine rational approximation of real numbers (continued fractions, Farey sequences or geometry of numbers), we know that there exist infinitely many couples $(x,y)\in\Z^2$ with 
$$
y\ge Y\quad \mbox{and} \quad y|x-\lambda_2^a y|\leq \frac{1}{2}\cdotp
$$
 For these couples $(x,y)$ of integers, we have 
$$\left\{ 
\begin{array}{l}
|x-\lambda_0^a y|\le 
|x-\lambda_2^a y|+ |\lambda_0^a-\lambda_2^a| y
\le
\displaystyle
\frac{1}{2y}+ |\lambda_0^a-\lambda_2^a| y
\le
y\left( 1+\frac{\epsilon}{3}\right) |\lambda_0^a-\lambda_2^a|, 
 \\[5mm]
|x-\lambda_1^a y|\le
|x-\lambda_2^a y|+ |\lambda_2^a-\lambda_1^a| y
\le
\displaystyle
\frac{1}{2y}+ |\lambda_2^a-\lambda_1^a| y
\le
y\left( 1+\frac{\epsilon}{3}\right) |\lambda_2^a-\lambda_1^a|, 
\end{array}\right.
$$
whereupon we have 
$$ 
|F_{n,a}(x,y)| =
|x-\lambda_0^a y|\cdot|x-\lambda_1^a y|\cdot |x-\lambda_2^a y|
\le 
 \frac{1}{2}y (1+\epsilon) 
|\lambda_2^a-\lambda_1^a|\cdotp
|\lambda_0^a-\lambda_2^a|
\cdotp 
$$
 \end{proof}

\section{The notation $\boldmath\BigOh$}\label{section:bigoh}

When $U$, $V$, $W$ are real numbers depending upon the data of the problem that we deal with, namely $n$, $a$, $x$, $y$ and $m$, the notations 
$$
U= V + \BigOh(W)
\quad\hbox{or}\quad 
U- V = \BigOh(W)
$$
mean that there exists an absolute positive constant $\kappa$
such that 
$$
|U-V|\le \kappa |W|
$$
for $\max\{|n|,|a|\}$ sufficiently large. These absolute constants $\kappa$, together with the absolute constant defined by the minimal value of 
$\max\{|n|,|a|\}$ from which all these inequalities hold true, are effectively computable. 
\smallskip

We also write 
$$
U= V\bigl(1 + \BigOh(W)\bigr)
$$
for 
$$
U= V + \BigOh(V W). 
$$

Very often the following lemma will be used without being explicitly quoted.

\smallskip

\begin{lemme}\label{Lemme:BigOh}
Let $U$, $U_1$, $U_2$, $V$, $V_1$, $V_2$, $W$, $W_1$, $W_2$ be real functions of $n$ and $a$. Suppose that $V$, $V_1$, $V_2$ do not vanish and that, when $\max\{|n|,|a|\}$ goes to infinity, the limits of $W$, $W_1$, $W_2$ 
all are $0$. 

{\rm (i)} Suppose 
$$
U= V\bigl(1 + \BigOh(W_1)\bigr)
$$
and $W_1\le W_2$. 
Then
$$
U= V\bigl(1 + \BigOh(W_2)\bigr).
$$

{\rm (ii)}
 Suppose 
$$
U= V\bigl(1 + \BigOh(W)\bigr).
$$
Then for $\max\{|n|,|a|\}$ sufficiently large, $U\not = 0$ and one has 
\begin{equation}\label{Equation:V=U+bigoh}
V= U\bigl(1 + \BigOh(W)\bigr),
\end{equation}\vskip -0.8cm
\begin{equation}\label{Equation:-U=-V+bigoh}
-U= -V\bigl(1 + \BigOh(W)\bigr),
\end{equation}\vskip -0.8cm
\begin{equation}\label{Equation:1/U=1/V+bigoh}
\frac{1}{U}= \frac{1}{V}\bigl(1 + \BigOh(W)\bigr),
\end{equation}\vskip -0.8cm
\begin{equation}\label{Equation:logU=logV+bigoh}
\log |U| = \log |V| +\BigOh(W). 
\end{equation} 

{\rm (iii)} Suppose 
$$
U_1= V_1\bigl(1 + \BigOh(W_1)\bigr)
\quad\hbox{and}\quad
U_2= V_2\bigl(1 + \BigOh(W_2)\bigr).
$$
Let $W=\max\{|W_1|, |W_2|\}$. Then
$$
U_1U_2= V_1V_2\bigl(1 + \BigOh(W)\bigr).
$$
If in addition, $V_1$ and $V_2$ have the same sign, then
$$
U_1+U_2= (V_1+V_2)\bigl(1 + \BigOh(W)\bigr).
$$ 

 {\rm (iv)} Suppose 
$$
U= V_1\bigl(1 + \BigOh(W)\bigr)
\quad\hbox{and}\quad
V_1= V_2\bigl(1 + \BigOh(W)\bigr).
$$
Then
$$
U= V_2\bigl(1 + \BigOh(W)\bigr).
$$
 \end{lemme} 

\begin{proof}[\indent Proof]
Since only the absolute values of $W$, $W_1$, $W_2$ come into play, we will assume that these values are $\ge 0$. 
The part (i) follows from the definition of $\BigOh$. 

\smallskip

The hypothesis in (ii) implies the existence of an absolute constant $\kappa>0$ such that for $\max\{|n|,|a|\}$ sufficiently large we have 
\begin{equation}\label{Equation:Hypothese(b)}
V-\kappa |V| \, W\leq 
U \le V+\kappa |V| \, W.
\end{equation}
The number $t=\kappa W$ verifies $0\le t<1$, whereupon, for $\max\{|n|,|a|\}$ sufficiently large, $U\not=0$, $U$ and $V$ are of the same sign and 
\begin{equation}\label{Equation:UsurV}
1-\kappa W\le 
\frac{U}{V}\le 1+\kappa W.
\end{equation}
For $0\leq t\leq 1/2$, one has 
$$
1-t \leq \frac{1}{1+t} \leq 1
\quad\hbox{and}\quad
1 \leq \frac{1}{1-t} \leq 1+2t.
$$
For $\max\{|n|,|a|\}$ sufficiently large, the number $t=\kappa W$ verifies $0\leq t\leq 1/2$; hence
$$1-\kappa W
\leq
\frac{1}{1+\kappa W}\le 
\frac{V}{U}\le 
\frac{1}{1-\kappa W}\le 1+2\kappa W
$$
and
$$
U-2\kappa |U| \, W \le V\le U+2\kappa |U| \, W,
$$ 
which secures the equality $(\ref{Equation:V=U+bigoh})$.
The proof of $(\ref{Equation:-U=-V+bigoh})$ is direct.

\smallskip

Upon division of each member of $(\ref{Equation:Hypothese(b)})$ by $UV$ (which is positive), we obtain 
$$
\frac{1}{U}-\kappa \frac{W}{|U|} \le 
\frac{1}{V} \le 
\frac{1}{U}+\kappa \frac{W}{|U|}, 
$$ 
which leads to $(\ref{Equation:1/U=1/V+bigoh})$. 

\smallskip
For $0\le t\le 1/2$, we have 
$$
-2t\le \log (1-t)
\quad\hbox{and}\quad
\log(1+t)\le 2t.
$$
Thanks to $(\ref{Equation:UsurV})$,
these inequalities with $t=\kappa W$ imply $(\ref{Equation:logU=logV+bigoh})$. This takes care of the proof of (ii). 

\smallskip

Using (i), we deduce from the hypothesis in (iii) that there exists a constant $\kappa$ such that
\begin{equation}\label{Equation:U1V1U2V2}
V_1-\kappa |V_1| \, W\le 
U_1 \le V_1+\kappa |V_1| \, W
\quad\hbox{and}\quad
V_2-\kappa |V_2| \, W\le 
U_2 \le V_2+\kappa |V_2| \, W.
\end{equation}
To estimate the value of the product, by using $(\ref{Equation:-U=-V+bigoh})$, we are reduced to the case where $U_1$ and $U_2$ are both positive (and so will be $V_1$ and $V_2$). 
Multiplying the members of these inequalities, we obtain
$$
(V_1-\kappa V_1W)(V_2-\kappa V_2W)\le 
U_1 U_2 \le 
(V_1+\kappa V_1W)( V_2+\kappa V_2W).
$$
It only remains to use the upper bound $3\kappa V_1V_2W $ for 
$2 \kappa V_1V_2W + \kappa^2V_1V_2W^2$ to conclude since $W$ goes to $0$. 

\smallskip
For the sum, supposing that $V_1$ and $V_2$ have the same sign (which is also the sign of $U_1$ and $U_2$), we add the corresponding members of the inequalities $(\ref{Equation:U1V1U2V2})$: 
$$
V_1+V_2-\kappa( |V_1|+|V_2|)W \le 
U_1 +U_2 \le V_1+V_2+\kappa (|V_1|+|V_2|)W;
$$
 note that $| V_1+V_2|=|V_1|+|V_2|$. So the part (iii) is secured. 

\smallskip
Notice that 
the hypothesis that $V_1$ and $V_2$ have the same signs may not be removed: consider for instance $V_1=-V_2$, $U_1=V_1$, $U_2=V_2+W$. 

\smallskip
Under the hypotheses of (iv), there exists an absolute constant $\kappa>0$ such that, for $\max\{|n|,|a|\}$ sufficiently large, we have 
$$
|U-V_1| \le \kappa |V_1|\, W
\quad\hbox{and}\quad
 |V_1-V_2| \le \kappa |V_2|\, W.
$$
Hence 
\begin{align}
\notag
|U-V_2| &\le \; |U-V_1|+|V_1-V_2|
\le \kappa (|V_1|+ |V_2|) W
\\[2mm]
\notag
&
\le \kappa |V_2|\, W + \kappa (|V_2|+\kappa |V_2|\, W)W
\le 3\kappa |V_2|\, W,
\end{align}
since $\kappa W\le 1$ for $\max\{|n|,|a|\}$ sufficiently large. 
This completes the proof of Lemma $\ref{Lemme:BigOh}$.
\end{proof}

 \section{Estimates for $\boldmath|x-\lambda_i^a y|$}
 
 Suppose that 
$n$, $a$, $m$, $x$, $y$ are rational integers with $n\ge 0$, $a\ge 2$, $m\ge 1$ and $y\ge 0$ 
such that 
$$
|F_{n,a}(x,y)|=m.
$$ 
When $n$ and $a$ are bounded by absolute constants, the conclusion of the parts (iii) and (iv) of Theorem $\ref{Theoreme:principal}$ is immediate, while the upper bounds for $\max\{|x|,|y|\}$ given in the parts (i) and (ii) of Theorem $\ref{Theoreme:principal}$ come from the theorem of Thue under an effective formulation (given by N.I.~Feldman and A.~Baker) that one can find for instance in Theorem 5.1 of \cite{ST}.
Therefore, let us suppose $\max\{n,a\}$ sufficiently large, so we are free to use the notation $\BigOh$ of Section $\ref{section:bigoh}$. Our paper \cite{LWrang1} takes care of the case where $n$ is bounded by an absolute constant. This allows us to suppose $n\ge 1$. 

\smallskip

Thanks to $(\ref{equation:boundslambda})$ and to our assumption $n\ge 1$, we have $|\lambda_2|<\lambda_0$.
 We deduce 
$$\left\{
\begin{array}{lll} 
\lambda_0^a-\lambda_1^a &=&\lambda_0^a\left(1+\BigOh\left(\displaystyle \frac{1}{\lambda_0^{2a}}\right)\right),
\\[5mm]
\lambda_0^a-\lambda_2^a
&=&\lambda_0^a\left(
1+\BigOh\left(\displaystyle
\frac{\lambda_2^a}{\lambda_0^a}\right)\right),
\\[5mm]
\lambda_2^a-\lambda_1^a
&=&\lambda_2^a\left(
1+\BigOh\left(\displaystyle
\frac{1}{\lambda_0^a}
\right) \right), 
\end{array} \right.
$$
 from which we obtain
$$\left\{
\begin{array}{lll} 
\displaystyle \frac{\lambda_{1}^a-\lambda_{2}^a}
{ \lambda_{1}^a-\lambda_{0}^a}
&=&
\displaystyle \frac{ (-1)^a (\lambda_0+1)^a} { \lambda_0^{2a}} 
 \left( 1+ \BigOh
 \left(
\displaystyle \frac{1}{\lambda_0^a}
\right)
\right),
 \\[5mm]\displaystyle
\frac{ \lambda_{2}^a-\lambda_{0}^a}
{\lambda_{2}^a-\lambda_{1}^a}
&=&\displaystyle
\frac{ (-1)^{a+1} \lambda_0^{2a}}{ (\lambda_0+1)^a}
\left( 1+ \BigOh
\left(\displaystyle
\frac{\lambda_2^a}{\lambda_0^a}
\right)
\right),
 \\[5mm]\displaystyle
\frac{ \lambda_{0}^a-\lambda_{1}^a}
{\lambda_{0}^a-\lambda_{2}^a}
&=& 1+ \BigOh
\left(\displaystyle
\frac{\lambda_2^a}{\lambda_0^a}\right).
\end{array} \right.
 $$

Writing
$$
\gamma_{0,a}=x-\lambda_0^a y, \quad
\gamma_{1,a}=x-\lambda_1^a y, \quad
\gamma_{2,a}=x-\lambda_2^a y,
$$
 we have 
\begin{equation}\label{Equation:produitgamma}
|\gamma_{0,a}\gamma_{1,a}\gamma_{2,a}|=m.
\end{equation}
For at least one index $ i \in \{0, 1, 2\}$, we have 
$$
|\gamma_{i,a}|\le \root 3 \of m
$$
and such an index will be denoted by $i_0$. The two other indices will be denoted by $i_1,\, i_2$, 
under the assumption that $(i_0,i_1,i_2)$ be a circular permutation of $(0,1,2)$. 

\smallskip
Let us first consider an easy case of the part (ii) of Theorem $\ref{Theoreme:principal}$, namely when we are under 
the hypothesis that $y\le 2\root 3 \of m$. It turns out that we have 
$$
|x|\le |\gamma_{i_0,a}|+ y |\lambda_{i_0}^a| \le \root 3 \of m + y \lambda_0^a,
$$
hence 
$$
\log 
\max\{|x|,y,2\}\le \frac{1}{3} \log m + (a+1) \log (n+2),
$$
which is much stronger than required to secure the conclusion of the part (ii) of Theorem $\ref{Theoreme:principal}$ . 

\smallskip
When $y=0$ or $y=1$, the first case of Theorem $\ref{Theoreme:principal}$ is immediate. As a matter of fact, the hypothesis 
$\max\{|x|, |y|\}\ge 2$ forces $|x|\ge 2$ and each of the two cases $y=0$ and $y=1$ are dealt with in some elementary way.

\smallskip

From now on and up to the end of section $\ref{Section:DemonstrationTheoremeagrand}$, we will assume that we are under the hypothesis 
\begin{equation}\label{equation:yge2root3ofm}
y\ge 2\root 3\of m.
\end{equation}

\smallskip

\begin{lemme}\label{Lemme:Estimations}$\,$
{\rm (i)} Suppose $i_0=0$. We have 
 $$
\gamma_{1,a}= y \lambda_0^a\left( 1+ \BigOh\left(
\frac{1}{\lambda_0^{2a}}
\right)
\right), \qquad
\gamma_{2,a} = y \lambda_0^a \left( 1+ \BigOh
\left(
\frac{\lambda_2^a}{\lambda_0^a}
\right)
\right), 
$$
$$
|\gamma_{0,a}|= \frac{m}{y^2 \lambda_0^{2a}} \left( 1+ \BigOh
\left(
\frac{\lambda_2^a}{\lambda_0^a}
\right)
\right),
$$
from which we obtain 
$$
\left|
\frac{\gamma_{0,a}(\lambda_{1}^a-\lambda_{2}^a)}
{\gamma_{2,a}(\lambda_{1}^a-\lambda_{0}^a)}
\right| 
=
\frac{ m (\lambda_0+1)^a} {y^3 \lambda_0^{5a}} 
 \left( 1+ \BigOh
 \left(
\frac{\lambda_2^a}{\lambda_0^a}
\right)
\right).
 $$ 
 
{\rm (ii)} 
Suppose $i_0=1$. We have 
 $$
\gamma_{2,a}= -y \lambda_2^a \left( 1+ \BigOh
\left(
\frac{1}{\lambda_0^a}
\right)
\right),
\qquad
\gamma_{0,a}= -y \lambda_0^a\left( 1+ \BigOh
\left(
\frac{1}{\lambda_0^{2a}}
\right)
\right),
$$
$$
|\gamma_{1,a}|= \frac{ m}{y^2 (\lambda_0+1)^a} \left( 1+ \BigOh
\left(
\frac{1}{\lambda_0^a}
\right)
\right),
$$
from which we obtain 
$$
\left|
\frac{\gamma_{1,a}(\lambda_{2}^a-\lambda_{0}^a)}
{\gamma_{0,a}(\lambda_{2}^a-\lambda_{1}^a)}
\right| 
=
\frac{ m\lambda_0^a}{y^3 (\lambda_0+1)^{2a}}
\left( 1+ \BigOh
\left(
\frac{\lambda_2^a}{\lambda_0^a}
\right)
\right).
$$ 

{\rm (iii)}
Suppose $i_0=2$. We have 
$$
\gamma_{0,a}= -y \lambda_0^a \left( 1+ \BigOh
\left(
\frac{\lambda_2^a}{\lambda_0^a}
\right)
\right),
\qquad
\gamma_{1,a}= y \lambda_2^a \left( 1+ \BigOh
\left(
\frac{1}{\lambda_0^a}
\right)
\right),
$$
 $$
|\gamma_{2,a}|=\frac{ m}{y^2 (\lambda_0+1)^a} \left( 1+ \BigOh
\left(
\frac{\lambda_2^a}{\lambda_0^a}
\right)
\right),
$$
from which we obtain
$$
\left|
\frac{\gamma_{2,a}(\lambda_{0}^a-\lambda_{1}^a)}
{\gamma_{1,a}(\lambda_{0}^a-\lambda_{2}^a)}
\right| 
=
\frac{ m \lambda_0^a}{y^3 (\lambda_0+1)^{2a}} 
\left( 1+ \BigOh
\left(
\frac{\lambda_2^a}{\lambda_0^a}
\right)
\right).
$$
\end{lemme} 
 
\begin{proof}[\indent Proof] 
We will use the equality
\begin{equation}\label{Equation:inegalitetriangulaire}
\gamma_{j,a}=x-\lambda_j^ay=\gamma_{i_0,a}+(\lambda_{i_0}^a-\lambda_j^a)y
\end{equation}
to estimate $\gamma_{0,a}$, $\gamma_{1,a}$ and $\gamma_{2,a}$. 

\smallskip
{\rm (i)} 
Suppose $i_0=0$. Using the hypotheses 
$(\ref{equation:yge2root3ofm})$ 
and $|\gamma_{i_0,a}|\le \root 3 \of m$, 
we deduce from $(\ref{Equation:inegalitetriangulaire})$ 
$$
\gamma_{1,a}= y\lambda_0^a\left(1+\BigOh
\left(
\frac{1}{\lambda_0^{a}}
\right)
\right)
\quad\hbox{and}\quad
\gamma_{2,a}= y\lambda_0^a\left(1+\BigOh
\left(
\frac{\lambda_2^a}{\lambda_0^a}
\right)
\right)
$$
and in this case the estimate for $|\gamma_{0,a}|$ results from 
$(\ref{Equation:produitgamma})$. 
Finally, we deduce the refined estimate 
$$
\gamma_{1,a}= y\lambda_0^a\left(1+\BigOh
\left(
\frac{1}{\lambda_0^{2a}}
\right)
\right)
$$
from $(\ref{Equation:inegalitetriangulaire})$ again.

\smallskip

{\rm (ii)} Suppose $i_0=1$. We have 
$$
\gamma_{0,a}= -y\lambda_0^a\left(1+\BigOh
\left(
\frac{1}{\lambda_0^a}
\right)
\right)
\quad \mbox{and} \quad 
\lambda_2^a-\lambda_1^a = \lambda_2^a \left(1+\BigOh
\left(\frac{1}{
\lambda_0^a}\right)\right).
$$
From the inequalities 
$$
|\gamma_{1,a}|\le \root 3 \of m \le \frac{y}{2}
\quad
 \hbox{and}
 \quad
 |\lambda_2^a-\lambda_1^a| \ge \frac{3}{4},
 $$
 we deduce, using $(\ref{Equation:inegalitetriangulaire})$, 
 $$
 |\gamma_{2,a}|\ge \frac{1}{4}y; 
 $$
 then the use of $(\ref{Equation:produitgamma})$ gives, for sufficiently large $\max\{n,a\}$, 
$$
 |\gamma_{1,a}|\le \frac{5m}{y^2\lambda_0^a}\le 
 \frac{5\root 3 \of m}{4 \lambda_0^a}\cdotp
 $$
Let us redo the same calculations, though this time, in $(\ref{Equation:inegalitetriangulaire})$, we replace the initial upper bound $|\gamma_{1,a}|\le \root 3 \of m $ by this last improved bound. This leads on the one hand to 
 $$
\gamma_{2,a}= -y \lambda_2^a \left( 1+ \BigOh
\left(
\frac{1}{\lambda_0^a}
\right)
\right). 
$$
On the other hand, thanks to $(\ref{Equation:produitgamma})$, we have
$$
|\gamma_{1,a}|= \frac{ m}{y^2 (\lambda_0+1)^a} \left( 1+ \BigOh
\left(
\frac{1}{\lambda_0^a}
\right)
\right).
$$ 
 Using this last upper bound in 
$(\ref{Equation:inegalitetriangulaire})$, we are led to the estimate for $|\gamma_{1,a}|$.

\smallskip
 {\rm (iii)} 
Suppose that $i_0=2$. The proof of this part mimics the proof of the part (ii). The estimate for $\gamma_{0,a}$ follows from 
$(\ref{Equation:inegalitetriangulaire})$. 
Next we use the inequalities
$$
|\gamma_{2,a}| \leq \frac{y}{2}
\quad\hbox{and}\quad |\lambda_2^a -\lambda_1^a| \geq \frac{3}{4}
$$
to obtain, for sufficiently large $\max\{n,a\}$, 
$$
|\gamma_{2,a}| \leq \frac{5\root 3 \of m}{4\lambda_0^a}\cdotp
$$
Using this last upper bound in 
$(\ref{Equation:inegalitetriangulaire})$, we are led to the estimate for 
$\gamma_{1,a}$ given in the lemma and we use it with $(\ref{Equation:produitgamma})$ to obtain the estimate for $|\gamma_{2,a}|$.
\end{proof} 

Note that we have 
$$
|\gamma_{i_0,a}| \le \left\{
\begin{array}{lll}
\displaystyle
\frac{2m}{y^2 \lambda_0^{2a}} 
&\mbox{if $\; i_0=0$},
\\ [5mm]
\displaystyle
\frac{2m}{y^2 ( \lambda_0+1)^a} 
&\hbox{if $\; i_0\in \{1, 2\}$}.
\\ 
\end{array}\right.
$$

\section{Rewriting an element of norm $\boldmath m$}

Any pair of elements among $\{\lambda_0,\, \lambda_1,\,\lambda_2\}$ is a fundamental system of units for the cubic field $\Q(\lambda_0)$ (p.~237 of \cite{T}). 

Using a result of Mignotte, Peth\H{o} and Lemmermeyer, we can write an element of norm $m$ 
in a way that we have a better control on it. The result is obtained as 
 a consequence of Lemma 3 of \cite{MPL} and it reads as follows.

\smallskip

\begin{lemme}\label{Lemme:MPL}
Let $\gamma$ be a nonzero element of $\Z[\lambda_0]$ of norm $\Norm(\gamma)=m$. Then there exist some integers $A$, $B$ and some nonzero element $\delta$ of $\Z[\lambda_0]$, with conjugates $\delta_0=\delta$, $\delta_1$ and $\delta_2$, such that
$$
\gamma=\delta \lambda_0^A\lambda_2^B
$$
with, for any $n\ge 3$, 
$$
\frac{\root 3 \of m}{\sqrt{n+3}}
\le | \delta_i|\le 
 \sqrt{n+3} \root 3 \of m \quad \mbox{for }i\in \{1,2\}
\quad
\mbox{and} \quad
\frac{\root 3 \of m}{n+3} \le 
|\delta_0|\le (n+3) \root 3 \of m.
$$
Moreover, if $m=1$, then $|\delta_0|=|\delta_1|=|\delta_2|=1$.
\end{lemme}

The statement of the last lemma follows from Lemma 3 of \cite{MPL} by taking 
$$
c_1=c_2=\frac{\root 3 \of m}{\sqrt{n+3}}\cdotp
$$
As in \cite{T}, we use $\lambda=\lambda_0$, $\lambda_1$ and $\lambda_2$, but 
 in \cite{MPL} these elements correspond respectively to $\lambda^{(3)}$, $\lambda^{(1)}$ and $\lambda^{(2)}$.
 
 \smallskip
The last estimations for $\delta_i$ with $i \in \{1,2\}$ and for $\delta_0$ lead to the inequalities
$$
\left|
\log |\delta_i|-\frac{1}{3}\log m
\right|
\le 
\frac{1}{2}\log (n+3)
\quad\hbox{and}\quad
\left|
\log |\delta_0|-\frac{1}{3}\log m
\right|
\le 
 \log (n+3).
$$
Moreover, we have 
$$
\prod_{i=1}^3 \max\{1,|\delta_i|\}\le (n+3)^2 m;
$$
since $\delta$ is an algebraic integer, we deduce 
\begin{equation}\label{Equation:Hauteur(delta)} 
\rmh(\delta)\le \frac{2}{3} \log (n+3)+ \frac{1}{3} \log m.
\end{equation}

In the case $m=1$, the conclusion of the lemma holds with $\delta=\pm 1$, since $\{\lambda_0, \lambda_2\}$ is a fundamental system of units for the ring $\Z[\lambda_0]$.

\section{Some estimations on the integers $\boldmath A$ and $\boldmath B$}\label{Section:IntroductionAetB}

We will use $\{\lambda_0, \lambda_2\}$ for a fundamental system of units for the cubic field $\Q(\lambda_0)$. Note that we have 
\begin{equation}\label{Equation:loglambda2} 
\left|
\log|\lambda_2| - \frac{1}{\lambda_0}\right|\le 
\frac{1}{2\lambda_0^2}
\end{equation}
and 
$$
\log|\lambda_1| =-\log\lambda_0-\log|\lambda_2|.
$$
The relation $(\ref{Equation:produitgamma})$ and Lemma $\ref{Lemme:MPL}$ indicate that there exist
 rational integers $A$, $B$ and an element $\delta_0\in\Z[\lambda_0]$, with conjugates $\delta_1$ and $\delta_2$, verifying the conclusion of Lemma $\ref{Lemme:MPL}$, such that, via Galois actions, we have 
\begin{equation}
\left\{
\begin{array}{rcl}
\gamma_{0,a}&=&\delta_0 \lambda_0^A \lambda_2^B,\\[2mm]
\gamma_{1,a}&=&\delta_1 \lambda_1 ^A\lambda_0^B\;
= \;
\delta_1
\lambda_0^{-A+B}
\lambda_2^{-A},\\[2mm]
\gamma_{2,a}&=&\delta_2 \lambda_2^A \lambda_1 ^B\;
=\;\delta_2
\lambda_0^{-B}
\lambda_2^{A-B}. 
\end{array} \right.
\end{equation}

Let us estimate $A$ and $B$. 
Writing $c_i=\log |\gamma_{i,a}|-\log |\delta_i|$ (for $i=0,1,2$), we obtain 
$$\left\{
\begin{array} {rcl}
A\log \lambda_0+B\log|\lambda_2|&=&c_0,
\\[2mm]
A\log |\lambda_1|+B\log \lambda_0&=&c_1,
\\[2mm]
A\log |\lambda_2|+B\log|\lambda_1|&=&c_2.
\end{array}\right. 
$$
The first two equations suffice to find the values $A$ and $B$.
The determinant $R$, where
$$
R= (\log \lambda_0)^2-(\log|\lambda_1| )(\log|\lambda_2|)\ge (\log \lambda_0)^2,
$$
is not zero.
 We have 
$$
A=\frac{1}{R}(c_0\log \lambda_0-c_1\log|\lambda_2|)\quad
\hbox{and}\quad 
B=-\frac{1}{R}(c_0 \log |\lambda_1|-c_1\log \lambda_0),
$$
which, by taking into account the relation $\lambda_0\lambda_1\lambda_2=1$, we write as
\begin{equation}
A=\frac{1}{R}\bigl(c_0\log\lambda_0-c_1\log|\lambda_2|\bigr)\quad
\hbox{and}\quad 
B=\frac{1}{R}\bigl((c_0+c_1)\log \lambda_0+c_0\log|\lambda_2|\bigr).
\end{equation}
The following estimates for 
$$
c_0=\log |\gamma_{0,a}|-\log |\delta_0|
\quad
\hbox{and}\quad
c_1=\log |\gamma_{1,a}|-\log |\delta_1|
$$
result from Lemma $\ref{Lemme:Estimations}$. 

\smallskip
\begin{lemme}\label{lemme:c0etc1}
$\;$
{\rm (i)} Suppose $i_0=0$. Then we have 
$$ \left\{ 
\begin{array}{lll}
c_0&=& 
\log m -2\log y-2a\log\lambda_0-\log |\delta_0|+
\BigOh
\left( \displaystyle 
\frac{\lambda_2^a}{\lambda_0^a}
\right),\\[3mm]
c_1&=&
\log y+a\log\lambda_0-\log |\delta_1|+
\BigOh\left(\displaystyle
\frac{1}{\lambda_0^{2a}}
\right).\\
\end{array} \right.
$$ 

 {\rm (ii)} Suppose $i_0=1$. Then we have 
$$
\left\{ 
\begin{array}{lll}
c_0&=&\log y+ a\log\lambda_0-\log |\delta_0|+
\BigOh\left( \displaystyle
\frac{1}{\lambda_0^{2a}}
\right),\\[3mm]
c_1&=&\log m -2\log y - a\log\lambda_0 - a\log|\lambda_2| -\log |\delta_1| 
+ \BigOh\left(\displaystyle
\frac{1}{\lambda_0^a}
\right).\\
\end{array} \right.
$$ 

 {\rm (iii)} Suppose $i_0=2$. Then we have 
$$
\left\{ 
\begin{array}{lll}
c_0&=& 
\log y +a\log \lambda_0-\log|\delta_0|+
\BigOh\left( \displaystyle
\frac{\lambda_2^a}{\lambda_0^a}
\right),\\[3mm]
c_1&=&
\log y +a\log |\lambda_2|-\log|\delta_1|+
\BigOh\left( \displaystyle
\frac{1}{\lambda_0^a}\right).
\end{array} \right.
$$
\end{lemme}

From Lemma $\ref{lemme:c0etc1}$ we deduce the following.

\smallskip

\begin{lemme}\label{Lemme:majorationMaxAB}
We have
$$
|A|+|B|
\le \Newcst{kappa:majorationMaxAB} \left( \frac{\log y 
}{\log \lambda_0}+ a \right).
$$
\end{lemme}

\section{The Siegel equation}

The Siegel equation 
$$
\gamma_{i_0,a}(\lambda_{i_1}^a-\lambda_{i_2}^a)+
\gamma_{i_1,a}(\lambda_{i_2}^a-\lambda_{i_0}^a)+
\gamma_{i_2,a}(\lambda_{i_0}^a-\lambda_{i_1}^a)=0
$$
leads to the identity 
$$
\frac{\gamma_{i_1,a}(\lambda_{i_2}^a-\lambda_{i_0}^a)}
{\gamma_{i_2,a}(\lambda_{i_1}^a-\lambda_{i_0}^a)}
-1= -
\frac{\gamma_{i_0,a}(\lambda_{i_1}^a-\lambda_{i_2}^a)}
{\gamma_{i_2,a}(\lambda_{i_1}^a-\lambda_{i_0}^a)},
$$
which will be used later.
 From Lemma $\ref{Lemme:Estimations}$ we deduce the inequalities
\begin{equation}\label{Equation:majoration}
0<
\left|
\frac{\gamma_{i_1,a}(\lambda_{i_2}^a-\lambda_{i_0}^a)}
{\gamma_{i_2,a}(\lambda_{i_1}^a-\lambda_{i_0}^a)}
-1
\right|
\le 
\frac{ 2m 
}{ y^3 \lambda_0^a 
} \cdotp 
\end{equation}

\section{Switching from $\boldmath A$ and $\boldmath B$ to $\boldmath A'$ and $\boldmath B'$}

 Since 
$$ \left\{
\begin{array}{lll} \displaystyle
\frac{\gamma_{1,a}}{\gamma_{2,a}}
& =& \displaystyle\frac{\delta_1}{\delta_2} 
\lambda_0^{-A+2B}
\lambda_2^{-2A+B},
\\[5mm] \displaystyle
 \frac{\gamma_{2,a}}{\gamma_{0,a}}& 
 =& \displaystyle\frac{\delta_2}{\delta_0} 
\lambda_0^{-A-B} \lambda_2^{A-2B},
\\[5mm] \displaystyle\frac{\gamma_{0,a}}{\gamma_{1,a}} 
& =& \displaystyle\frac{\delta_0}{\delta_1} 
\lambda_0^{2A-B}
\lambda_2^{A+B},
\end{array} \right.
$$
 we are led to introduce 
 $$
 (A',B')=
 \begin{cases} 
 (-A+2B, -2A+B) & \mbox{ for $\; i_0=0$},
 \\[2mm]
 (-A-B, A-2B) & \mbox{ for $\; i_0=1$},
\\[2mm]
 (2A-B, A+B) & \mbox{ for $\; i_0=2$},
 \\[2mm]
 \end{cases}
 $$
so we can write 
 \begin{equation}\label{Equation:fll}
 \frac{\gamma_{i_1,a}(\lambda_{i_2}^a-\lambda_{i_0}^a)}
{\gamma_{i_2,a}(\lambda_{i_1}^a-\lambda_{i_0}^a)}
=\mu \lambda_0^{A'}
\lambda_2^{B'}
\end{equation}
 with
 $$
 \mu=
\frac{\delta_{i_1}}{\delta_{i_2}} \; 
\left( \frac{ \lambda_{i_2}^a-\lambda_{i_0}^a}
{ \lambda_{i_1}^a-\lambda_{i_0}^a}\right)\cdot
$$
Since
$$
\rmh(\lambda_0)=\frac{1}{3}\bigl(\log\lambda_0 + \log|\lambda_2|\bigr)=\frac{1}{3} \log(\lambda_0+1)
$$
and since, by the inequality $(\ref{Equation:Hauteur(delta)})$, 
$$
\rmh\left(\frac{\delta_{i_1}}{\delta_{i_2}} \right)\le 
2\rmh(\delta)\le \frac{4}{3} \log (n+3)+\frac{2}{3}\log m,
$$
we have 
$$
\rmh(\mu)\le 
\frac{4}{3} \log (n+3)+\frac{2}{3}\log m
+\frac{4}{3} a \log(\lambda_0+1)+2\log 2,
$$ 
whereupon we obtain
\begin{equation}\label{Equation:hauteurdemu}
\rmh(\mu)\le
3 \bigl(\log m + a \log (n+3)\bigr).
\end{equation} 
 By using $(\ref{Equation:fll})$, we write $(\ref{Equation:majoration})$ as
\begin{equation}\label{Equation:inegalitefondamentale}
0<\left|
 \mu
\lambda_0^{A'} \lambda_2^{B'}
-1
\right|
\le 
\frac{2 m 
}{ y^3 \lambda_0^a } \cdotp 
\end{equation}  
 
\section{Proof of the second part of Theorem $\ref{Theoreme:principal}$}
 
We are now ready to write the 

\smallskip
\begin{proof}[\indent Proof of the part (ii) of Theorem $\ref{Theoreme:principal}$.]
Let us write the right member of $(\ref{Equation:fll})$ in the form
$$
 \lambda_0^{A'} \lambda_2^{B'} \mu
=
\gamma_1^{c_1}\gamma_2^{c_2}\gamma_3^{c_3},
$$ 
with 
$$
\gamma_1=\lambda_0,\quad
\gamma_2=\lambda_2,\quad
\gamma_3=\mu,
\quad
c_1=A',\quad c_2=B', \quad c_3=1.
$$
The inequality $(\ref{Equation:hauteurdemu})$ provides an upper bound for the height of $\mu$. Note that 
$$
 |A'|+ |B'| \le 3 ( |A| + |B|).
 $$
Let us use Proposition 2 of \cite{LW} with 
$$ 
s=3,\quad 
H_1=H_2=\Newcst{kappa:H1etH2} \log n, 
\quad
H_3=3\cst{kappa:H1etH2} (\log m+a\log n), 
$$
$$
C=(|A|+|B|) \frac{ \log n}{\log m+a\log n} +2.
$$
 This gives
$$
\left|
\frac{\gamma_{i_1,a}(\lambda_{i_2}^a-\lambda_{i_0}^a)}
{\gamma_{i_2,a}(\lambda_{i_1}^a-\lambda_{i_0}^a)}
-1
\right| \ge \exp \left\{ -\Newcst{kappa:majfll1} (\log m+a\log n)
(\log n)^2 \log C \right\}.
$$
We deduce from $(\ref{Equation:majoration})$ the existence of a constant $\Newcst{kappa:majfll2}$ such that
\begin{equation}\label{Equation:logymaj}
 \log y \le 
\cst{kappa:majfll2}(\log m+a\log n) (\log n)^2\log C.
\end{equation}
Then Lemma $\ref{Lemme:majorationMaxAB}$ leads to 
$$
|A|+|B| \le 
\Newcst{kappa:majfll3} (\log m+a\log n) (\log n) \log C,
$$
hence
$$
C
 \le 
2\cst{kappa:majfll3} (\log n)^2 \log C,
$$
whereupon
$$
C
 \le 
\Newcst{kappa:Lev} (\log n)^2\log\log n.
$$ 
This leads to
\begin{equation}\label{Equation:majorationAetB}
|A|+|B|\le \cst{kappa:majfmajA+B} (\log m+a\log n)(\log n) 
\log \log n
\end{equation}
with $\Newcst{kappa:majfmajA+B}>0$. 
From $(\ref{Equation:logymaj})$ we deduce
$$
 \log y \le 
\Newcst{kappa:majfll7} (\log m+a\log n) (\log n)^2\log \log n.
$$
Finally, 
$$
|x|\le |\gamma_{i_0,a}|+ y |\lambda_{i_0}^a| \leq \root 3 \of m
 +y\lambda_0^a\leq \frac12 y +y\lambda_0^a \leq 
 2y \lambda_0^a,
$$
hence
$$
 \log |x| \le 
2\cst{kappa:majfll7} (\log m+a\log n) (\log n)^2\log \log n.
 $$
This secures the proof of the part (ii) of Theorem $\ref{Theoreme:principal}$. 
\end{proof}

\indent
\begin{remarque}
When we suppose 
$n\ge 3$, 
$$
m\le \frac{n}{(\log n)^3}
$$
and
\begin{equation}\label{Equation:apetit}
2\le a \le \frac{{n}}{(\log n)^4},
\end{equation} 
the upper bound $(\ref{Equation:majorationAetB})$ 
gives 
\begin{equation}\label{Equation:majorationAetBm=1}
|A|+|B|\le 2\cst{kappa:majfmajA+B} n \frac{
\log \log n}{(\log n)^2} \cdotp
\end{equation} 
\end{remarque}

\section{Estimations of $\boldmath \log y$ }\label{Section:estimationdey}

In this section, we combine 
Lemma $\ref{Lemme:Estimations}$ with the results of Section 
 $\ref{Section:IntroductionAetB}$;
 this allows to estimate $\log y$ in two different ways. A comparison between the two estimates will provide some relations, 
another proof of which being given in Lemma $\ref{Lemme:consequenceestimationy}$. 

\smallskip

 \begin{lemme}\label{Lemme:minorationy}
$\;$ {\rm (i)} If $i_0=0$, then we have 
$$
\begin{array}{lll}
\log y &=& \log |\delta_1| - (A-B+a) \log \lambda_0 - A \log |\lambda_2|
+ \BigOh 
\left( \displaystyle
\frac{1}{\lambda_0^{2a}}
\right),
\\[3mm]
&=& \log |\delta_2| - (B+a) \log \lambda_0 + (A-B)\log |\lambda_2|
+ \BigOh 
\left(\displaystyle
\frac{\lambda_2^a}{\lambda_0^a}
\right),
\end{array}
$$

{\rm (ii)} If $i_0=1$, then we have 
$$
\begin{array}{lll} 
\log y&=&
 \log |\delta_2| -B \log \lambda_0 + (A-B-a)\log |\lambda_2|
+ \BigOh
\left(\displaystyle
\frac{1}{\lambda_0^a}
\right),
\\[3mm]
&=&\log |\delta_0| + (A-a) \log \lambda_0 + B\log |\lambda_2|
+ \BigOh
\left(\displaystyle
\frac{1}{\lambda_0^{2a}}\right)
\end{array}
$$ 

{\rm (iii)} If $i_0=2$, then we have 
$$
\begin{array}{lll} 
\log y& =& \log |\delta_1| - (A-B)\log \lambda_0 - (A+a)\log |\lambda_2|
+ \BigOh
\left(\displaystyle
\frac{1}{\lambda_0^a}
\right)
\\[3mm]
 &=& \log |\delta_0| + (A-a)\log \lambda_0 + B \log |\lambda_2|
+ \BigOh
\left(\displaystyle
\frac{\lambda_2^a}{\lambda_0^a}\right)
\end{array}
$$ 
\end{lemme}

\begin{proof}[\indent Proof] {\rm (i)}
Suppose $i_0=0$. Then we have 
$$
\gamma_{1,a} = 
\delta_1 \lambda_0^{-A+B} \lambda_2^{-A} 
= y \lambda_0^a\left( 1+ \BigOh
\left(
\frac{1}{\lambda_0^{2a}}
\right)
\right),
$$
from which we deduce
$$
y = 
\delta_1 \lambda_0^{-A+B-a}
 \lambda_2^{-A}
 \left(1+ \BigOh
 \left(
\frac{1}{\lambda_0^{2a}}
\right)
 \right).
$$
We also have 
$$
\gamma_{2,a} =\delta_2 \lambda_0^{-B} \lambda_2^{A-B}=
 y \lambda_0^a \left( 1+ \BigOh\left(
\frac{\lambda_2^a}{\lambda_0^a}
\right)\right), 
$$ 
 hence
 $$
 y = 
\delta_2 \lambda_0^{-B-a}
 \lambda_2^{A-B}
 \left(1+ \BigOh
 \left(
\frac{\lambda_2^a}{\lambda_0^a}
\right) \right).
$$

 {\rm (ii)}
Suppose $i_0=1$. Then we have 
$$
\gamma_{2,a}= \delta_2 \lambda_0^{-B} \lambda_2^{A-B}
=
-y \lambda_2^a \left( 1+ \BigOh\left(
\frac{1}{\lambda_0^a}
\right)\right),
 $$
from which we deduce 
$$
y=- \delta_2 \lambda_0^{-B} \lambda_2^{A-B-a}
 \left(1+ \BigOh\
 \left(
\frac{1}{\lambda_0^a}
\right)
\right).
$$ 
We also have 
$$
\gamma_{0,a} = \delta_0 \lambda_0^A \lambda_2^B
=
-y \lambda_0^a\left( 1+ \BigOh
\left(
\frac{1}{\lambda_0^{2a}}
\right)
\right),
$$
hence
$$
y=- \delta_0 \lambda_0^{A-a} \lambda_2^{B}
 \left(1+ \BigOh\
 \left(
\frac{1}{\lambda_0^{2a}}
\right)
\right).
$$ 
 
{\rm (iii)}
Suppose $i_0=2$. Then we have 
 
$$
\gamma_{1,a}
=
 \delta_1 \lambda_0^{-A+B} \lambda_2^{-A}
 =
 y \lambda_2^a \left( 1+ \BigOh
 \left(
\frac{1}{\lambda_0^a}
\right)
\right),
$$
from which we deduce
$$
 y = \delta_1 \lambda_0^{-A+B} \lambda_2^{-A-a}
 \left(1+ \BigOh
 \left(
\frac{1}{\lambda_0^a}
\right)
\right).
 $$
 We also have 
 $$
\gamma_{0,a}
=
 \delta_0 \lambda_0^A \lambda_2^B
 =
 -y \lambda_0^a \left( 1+ \BigOh
 \left(
\frac{\lambda_2^a}{\lambda_0^a}
\right)
\right),
$$
hence
$$
 y =- \delta_0 \lambda_0^{A-a} \lambda_2^{B}
 \left(1+ \BigOh
 \left(
\frac{\lambda_2^a}{\lambda_0^a}
\right)
 \right).
 $$
\end{proof}
  
 \begin{lemme}\label{Lemme:consequenceestimationy}
 One has 
$$
\Lambda = 
 \begin{cases}
 \displaystyle
A'\log\lambda_0+B'\log|\lambda_2|+\log\frac{|\delta_1|}{|\delta_2|}+\BigOh
\left(
\frac{\lambda_2^a}{\lambda_0^a}
\right)
&
\hbox{if $\; i_0=0$, }
\\
\null
\\
 \displaystyle
(A'+a)\log\lambda_0+(B'-a)\log|\lambda_2|+\log\frac{|\delta_2|}{|\delta_0|}+
\BigOh
\left(
\frac{1}{\lambda_0^a}
\right)
&
\hbox{if $\;i_0=1$,}
\\
\null
\\
 \displaystyle
(A'-a)\log\lambda_0+(B'+a)\log|\lambda_2|+\log\frac{|\delta_0|}{|\delta_1|}
+\BigOh\left(
\frac{\lambda_2^a}{\lambda_0^a}
\right)
&\hbox{if $\; i_0=2$.}
\end{cases} 
$$
 \end{lemme}

\begin{proof}[\indent Proof]
 {\rm (i)} Suppose $i_0=0$. 
We just use 
$$
\Lambda=A'\log\lambda_0+B'\log|\lambda_2|+\log\frac{|\delta_1|}{|\delta_2|}+\log\frac{|\lambda_2^a - \lambda_0^a|}{|\lambda_1^a - \lambda_0^a|} 
$$
with
$$
 \log\frac{|\lambda_2^a - \lambda_0^a|}{|\lambda_1^a - \lambda_0^a|}=\BigOh
 \left(
\frac{\lambda_2^a}{\lambda_0^a}
\right).
 $$
 
 {\rm (ii)} Suppose $i_0=1$. 
This time, we use 
$$
\Lambda=A'\log\lambda_0+B'\log|\lambda_2|+\log\frac{|\delta_2|}{|\delta_0|}+\log\frac{|\lambda_0^a - \lambda_1^a|}{|\lambda_2^a - \lambda_1^a|} 
$$
with
$$
\log\frac{|\lambda_0^a - \lambda_1^a|}{|\lambda_2^a - \lambda_1^a|}
=a\log\lambda_0-a\log|\lambda_2|+\BigOh\left(
\frac{1}{\lambda_0^a}
\right).
$$ 

 {\rm (iii)} Suppose $i_0=2$. 
Here, we use 
$$
\Lambda=A'\log\lambda_0+B'\log|\lambda_2|+\log\frac{|\delta_0|}{|\delta_1|}+\log\frac{|\lambda_1^a - \lambda_2^a|}{|\lambda_0^a - \lambda_2^a| }
$$
with
$$
\log\frac{|\lambda_1^a - \lambda_2^a|}{|\lambda_0^a - \lambda_2^a|}
=-a\log\lambda_0+a\log|\lambda_2|+\BigOh\left(
\frac{\lambda_2^a}{\lambda_0^a}
\right).
$$ 
 \end{proof}
 
By using 
 $(\ref{Equation:loglambda2})$, $(\ref{Equation:majorationAetBm=1})$ and $(\ref{Equation:flllog})$, we deduce from Lemma $\ref{Lemme:consequenceestimationy}$ the following statement.

\smallskip

 \begin{corollaire}\label{corollaire:A-2B}
 Assume 
$\displaystyle
m\le \frac{n}{(\log n)^3}, \;
2\le a \le \frac{{n}}{(\log n)^4} 
$
and 
$n$ sufficiently large. 

\smallskip
\noindent {\rm (i)}
Suppose $i_0=0$. Then 
$|A-2B|\le 1$. Moreover, if $m=1$, then $A=2B$.

\smallskip

\noindent {\rm (ii)}
Suppose $i_0=1$. Then 
 $|A+B-a|\le 1$. If $m=1$, then $A+B=a$.
 
 \smallskip
 \noindent {\rm (iii)}
Suppose $i_0=2$. Then 
$|2A-B-a|\le 1$. If $m=1$, then $2A=B+a$.
 \end{corollaire}
 
 \begin{proof}[\indent Proof]
Thanks to Lemma $\ref{Lemme:consequenceestimationy}$, we have the 
following equalities: \\

\noindent $\bullet \;$ If $i_0=0$, then
$$
(A-2B)\log \lambda_0 + \log \frac{|\delta_2|}{|\delta_1|} + \frac {2A-B} {\lambda_0}
= \BigOh\left( 
\frac{|2A-B|+1}{\lambda_0^2}
\right).
$$
$\bullet \;$ 
If $\; i_0=1$, then
$$\displaystyle
(A+B-a)\log \lambda_0 + \log \frac{|\delta_0|}{|\delta_2|} - \frac{A-2B-a}{\lambda_0} 
= \BigOh\left( 
\frac{ |A-2B-a|+1}{\lambda_0^2}
\right).
$$
 $\bullet \;$ If $\; i_0=2$, then 
$$
\displaystyle 
(2A-B-a) \log \lambda_0 + \log \frac{|\delta_0|}{|\delta_1|} +\frac{A+B+a}{\lambda_0}
= \BigOh\left( 
\frac{|A+B+a|+1}{\lambda_0^2}
\right).
$$

In $\BigOh\displaystyle\left( \frac{|U|+1}{\lambda_0^2} \right)$ where $U$ is either $2A-B$, $A-2B-a$ or $A+B+a$, we included $+1$ in order to take into account the case $U=0$.
\medskip
 
The following inequalites, which come from Lemma $\ref{Lemme:MPL}$, will be used:
 $$
\left| \log \frac{|\delta_2|}{|\delta_1|}\right|
 \le \log (n+3),\quad 
 \left| \log \frac{|\delta_0|}{|\delta_2|}\right|
 \le \frac{3}{2} \log (n+3),
 \quad
 \left| \log \frac{|\delta_0|}{|\delta_1|}\right|
 \le \frac{3}{2} \log (n+3).
$$
Then we can count on $(\ref{Equation:majorationAetBm=1})$ and 
$(\ref{Equation:flllog})$ to claim that 
each of the terms 
$$
\frac {2A-B} {\lambda_0},
\quad
 \frac{A-2B-a}{\lambda_0},
 \quad
\frac {A+B+a} {\lambda_0},\quad 
\BigOh\left( 
\frac{|A|+|B|+a}{\lambda_0^2}
\right)
$$
goes to $0$ as 
$n$ goes to infinity. 
It happens that in each of the cases (i), (ii), (iii), we have a formula looking like 
$V\log\lambda_0 +\log \displaystyle \frac{|\delta_i|}{[\delta_j|}$, which goes to $0$.
Since $\log \lambda_0$ behaves like $\log n$, since $\log \displaystyle \frac{|\delta_i|}{[\delta_j|}$ is between $-(3/2)\log (n+3)$ and $+(3/2)\log (n+3)$ and since $V\in\Z$, we deduce that $V$ is between $-1$ and $+1$. 

\smallskip
Finally, in the case $m=1$, Lemma $\ref{Lemme:MPL}$ allows us to use the fact that 
 $|\delta_0|=|\delta_1|=|\delta_2|=1$.
 \end{proof}
 
 In taking the logarithms of the absolute values, we deduce from $(\ref{Equation:inegalitefondamentale})$ and from the assumption $(\ref{equation:yge2root3ofm})$
 that the number
$$
\Lambda=A'\log \lambda_0 + B'\log |\lambda_2|+\log| \mu|, \mbox{ with }\; \mu =
\frac{\delta_{i_1}}{\delta_{i_2}} \; 
\left( \frac{ \lambda_{i_2}^a-\lambda_{i_0}^a}
{ \lambda_{i_1}^a-\lambda_{i_0}^a}\right),
$$
verifies
\begin{equation}\label{Equation:flllog}
0< |\Lambda|
\le 
\frac{\cst{kappa:Equation:flllog}
 }{ \lambda_0^a }
\end{equation} 
with $\Newcst{kappa:Equation:flllog}>0$. 

As a consequence, we have
 \begin{equation}\label{Equation:formulesavecA',B'}
\begin{cases}
\displaystyle
A'\log \lambda_0 + \log \frac{|\delta_1|}{|\delta_2|} + B' \log|\lambda_2|
= \BigOh
\left(
\frac{\lambda_2^a}{\lambda_0^a}
\right)
&
\hbox{if $i_0=0$,}
\\
\null
\\
\displaystyle 
(-A'-a)\log \lambda_0 + \log \frac{|\delta_0|}{|\delta_2|} - (B'-a)\log|\lambda_2|
= \BigOh
\left(
\frac{1}{\lambda_0^a}
\right)
&
\hbox{if $i_0=1$,}
\\
\null
\\
\displaystyle 
(A'-a) \log \lambda_0 + \log \frac{|\delta_0|}{|\delta_1|} + (B'+a)\log|\lambda_2|
= \BigOh
\left(
\frac{\lambda_2^a}{\lambda_0^a}
\right)
&
\hbox{if $i_0=2$.}
\end{cases}
\end{equation} 

 These estimates also follow from Lemma $\ref{Lemme:minorationy}$ by using Lemma $\ref{Lemme:BigOh}$.

\section{Proof of the first part of Theorem $\boldmath \ref{Theoreme:principal}$ } \label{Section:DemonstrationTheorem1}
 
 Let us assume the hypotheses of Theorem $\ref{Theoreme:principal}$ with $m=1$ and let us also suppose 
$$
 2\le a \le \frac{{n}}{(\log n)^4} .
$$
We apply the previous results to the case $m=1$ 
by supposing that $n$ is sufficiently large, say $n\ge \Newcst{kappa:n0}$,
and by assuming that we have a solution with $y\geq 2$. 
Note that by Lemma $\ref{Lemme:MPL}$, we have $|\delta_0|=|\delta_1|=|\delta_2|=1$.
Depending upon the values of $i_0$, there are three cases to consider. 

\smallskip

{\rm (i)} Suppose $i_0=0$. By Corollary $\ref{corollaire:A-2B}$,
we have $A=2B$ and from $(\ref{Equation:formulesavecA',B'})$ we deduce 
$$
(-2A+B) \log|\lambda_2|
= \BigOh\left(\frac{\lambda_2^a}{\lambda_0^a}\right),
$$
Hence $2A=B$, and consequently $A=B=0$. 
Now Lemma $\ref{Lemme:minorationy}$ implies the contradiction
$$
\log y = -a \log \lambda_0 
+\BigOh\left(
\frac{1}{\lambda_0^{2a}}
\right).
$$
This is not possible since $y\ge 2$.

\smallskip
 
 {\rm (ii)} Suppose $i_0=1$.
 Corollary $\ref{corollaire:A-2B}$ gives $A+B=a$. From $(\ref{Equation:formulesavecA',B'})$ we deduce
$$
-(A-2B-a)\log|\lambda_2|
= \BigOh\left(
\frac{1}{\lambda_0^a}
\right).
$$
This last relation
 implies $A=2B+a$, hence $B=0$ and $A=a$. Then 
Lemma $\ref{Lemme:minorationy}$ implies 
$$ 
\log y = \BigOh\left(
\frac{1}{\lambda_0}
\right),
$$
which is not possible because $y\ge 2$.

\smallskip

 {\rm (iii)} Suppose $i_0=2$.
Thanks to Corollary $\ref{corollaire:A-2B}$ we have $2A-B=a$. 
From $(\ref{Equation:formulesavecA',B'})$, we deduce
$$
(A+B+a ) \log|\lambda_2|
= \BigOh\left(\frac{\lambda_2^a}{\lambda_0^a}\right),
$$
hence $A+B=-a$. Therefore $A=0$ and $B=-a$. Then
Lemma $\ref{Lemme:minorationy}$ implies 
$$
\log y = -a \log \lambda_0 -a \log |\lambda_2 |
+\BigOh\left(
\frac{1}{\lambda_0}
\right).
$$
This is not possible since $y\ge 2$.

\smallskip

This proves that under the hypotheses of the part (i) of Theorem $\ref{Theoreme:principal}$ in the case $(\ref{Equation:apetit})$, we have $n<\cst{kappa:n0}$. This implies that $a$ is also bounded because of $(\ref{Equation:apetit})$. 

 \section{Proof of the third part of Theorem $\boldmath \ref{Theoreme:principal}$ }\label{Section:DemonstrationTheoremeagrand}
 
From Lemma $\ref{Lemme:consequenceestimationy}$ we deduce the following. 

\smallskip

\begin{lemme}\label{Lemme:majorationfllcasagrand}
 One has 
 $$
 \begin{cases}
 \displaystyle
\left|
A'\log\lambda_0+B'\log|\lambda_2|+\log\frac{|\delta_1|}{|\delta_2|}
\right|
\le \frac{\Newcst{kappa:majorationfllcasagrand} |\lambda_2|^a}{ \lambda_0^{a}}
&
\hbox{if $\; i_0=0$, }
\\
\null
\\
 \displaystyle
\left|(A'+a)\log\lambda_0+(B'-a)\log|\lambda_2|+\log\frac{|\delta_2|}{|\delta_0|}
\right|
\le \frac{\cst{kappa:majorationfllcasagrand}}{ \lambda_0^{a}}
&
\hbox{if $\; i_0=1$,}
\\
\null
\\
 \displaystyle
\left|(A'-a)\log\lambda_0+(B'+a)\log|\lambda_2|+\log\frac{|\delta_0|}{|\delta_1|}
 \right|
\le \frac{\cst{kappa:majorationfllcasagrand} |\lambda_2|^a}{ \lambda_0^{a}}
&\hbox{if $\; i_0=2$.}
\end{cases}
$$
\end{lemme}
 
\begin{proof}[\indent Proof of the part (iii) of Theorem $\ref{Theoreme:principal}$]
 We will take advantage of Lemma $\ref{Lemme:majorationfllcasagrand}$ by using Proposition 2 of \cite{LW}, with 
$$ 
s=3,\quad 
H_1=H_2=\Newcst{kappa:H1etH2agrand} \log n, 
\quad
H_3=3\cst{kappa:H1etH2agrand} (\log m+\log n), 
$$
$$
C'=(|A|+|B|+a) \frac{ \log n}{\log m+\log n} +2.
$$ 
This proposition allows to exhibit a lower bound for the left member of the inequalities of Lemma $\ref{Lemme:majorationfllcasagrand}$, namely 
$$
\exp\{-
\Newcst{kappa:minorationfllagrand} 
(\log m+\log n)(\log n)^2 \log C'
\}. 
$$
We are led to
$$
a \le \Newcst{kappa:majorationa} (\log m+\log n)(\log n)\log C'.
$$
Thanks to $(\ref{Equation:majorationAetB})$ we also have 
$$
 C' 
 \le
(\log n)^3\left(\frac{a\log n}{\log m+\log n} +2\right),
 $$
 hence
 $$
 \frac{a\log n}{\log m+\log n}\le
\cst{kappa:majorationa}
 (\log n)^2
 \left(
 3 \log \log n+ \log 
 \left( \frac{a \log n}{\log m+\log n} +2\right)
 \right).
 $$
 This allows to write 
 $$
a\le \cst{kappa:agrand} (\log m+\log n) (\log n) \log\log n.
$$ 
This proves the part (iii) of Theorem $\ref{Theoreme:principal}$ and also completes the proof of the part (i) 
of Theorem $\ref{Theoreme:principal}$. \end{proof}

\section{Proof of the fourth part of Theorem $\boldmath \ref{Theoreme:principal}$ } \label{Section:Demonstration(iv)}

\begin{proof}[Proof of the part {\rm (iv)} of Theorem $\ref{Theoreme:principal}$]
We may assume 
 $$
 a\ge \cst{kappa:ysurxlambda2}, \quad
 y\ge 1, \quad
\max\left \{
 |x|,\; y 
\right\}^{ \cst{kappa:ysurxlambda2}}
\le 
e^{a}
\quad\hbox{and}\quad 
 4m \le \lambda_0^{a/2}
 $$ 
 with a sufficiently large constant $\Newcst{kappa:ysurxlambda2}$. Using the estimates 
 $$
|x-\lambda_0^a y|
\ge
 \lambda_0^a y -|x|
\ge
 \lambda_0^a y -\frac{1}{2} \lambda_0^a
\ge
\frac{1}{2} \lambda_0^a y 
$$
and
 $$
|x-\lambda_1^a y|\ge |x| - |\lambda_1|^a y\ge |x| -\frac{1}{2} \ge \frac{|x|}{2},
$$
we deduce from the relation 
$$
m=(x-\lambda_0^a y)(x-\lambda_1^a y)(x-\lambda_2^a y)
$$ 
the estimates
\begin{equation}\label{Equation:minorationysurxlambda2a}
0<\left| \lambda_2^a \, \frac{y}{x} -1\right |\le \frac{4m}{x^2 \lambda_0^{a}}\le \lambda_0^{-a/2}.
\end{equation}
Since $|\lambda_2|^a\ge 1$ and $\lambda_0^{-a/2}<1/2$, we deduce $y\le 2|x|$. 
We use Proposition 2 of \cite{LW} with 
$$ 
s=2,\quad 
\gamma_1=\lambda_2,
\quad 
\gamma_2=\frac{y}{x}, 
\quad c_1=a,\quad c_2=1,
$$
$$
H_1= \log (n+2),
\quad
H_2=1+ \log |x|, 
\quad 
C= \frac{ a \log (n+2)}{H_2} +2.
$$
 This gives
 $$
 \left| \gamma_1^{c_1} \gamma_2^{c_2} -1\right|
 \ge 
 \exp\bigl\{
 -\Newcst{kappa:ysurxlambda2a} H_1H_2\log C
 \bigr\}.
 $$
 Combining with ($\ref{Equation:minorationysurxlambda2a}$), we obtain
 $$
 \frac{C}{\log C}\le \Newcst{kappa:ysurxlambda2b} \log(n+2).
 $$
 Hence 
 $$
 a\le \Newcst{kappa:ysurxlambda2c}
 \bigl( 1+\log |x| \bigr) \log \log(n+3).
 $$
 This completes the proof of the part {\rm (iv)} of Theorem $\ref{Theoreme:principal}$.
\end{proof}
 
\section{Some numerical calculations}\label{Section:NumericalCalcutation}

 Fix an integer $n\geq 0$. For all integers $a\geq 0$ the cubic forms
$$
F_{n,a}(X,Y)=X^3-u_aX^2Y+(-1)^a v_aXY^2-Y^3
$$
can be explicitly written by using the recurrence formulas for $u_a$ and $v_a$, namely
$$ \left\{
\begin{array}{lll} 
u_{a+3}&=&(n-1)u_{a+2}+(n+2)u_{a+1}+u_a,\\[2mm]
v_{a+3}&=&(n+2)v_{a+2}-(n-1)v_{a+1}-v_a,
\end{array}\right.
$$
with the the initial conditions 
$$ \left\{
\begin{array}{lll} 
F_{n,0}(X,Y)&=&X^3-3X^2Y+3XY^2-Y^3,
\\[2mm]
F_{n,1}(X,Y)&=&X^3 -(n-1)X^2Y -(n+2)XY^2 -Y^3,
\\[2mm]
F_{n,2}(X,Y)&=&X^3 -(n^2+5)X^2Y+(n^2+2n+6)XY^2-Y^3.
 \end{array}\right.
$$
For instance,
$$\left\{
\begin{array}{lll}
F_{n,3}(X,Y)&=&X^3-\hfill(n^3+6n-4)X^2Y-(n^3+3n^2+9n+11)XY^2-Y^3,
\\[3mm]
F_{n,4}(X,Y)&=&X^3-(n^4+8n^2-4n+13)X^2Y
\\[0mm]
&&\hfill +(n^4+4n^3+14n^2+24n+26)XY^2-Y^3,
\\[3mm]
F_{n,5}(X,Y)&=&X^3-(n^5+10n^3-5n^2+25n-16)X^2Y
\\[0mm]
&&\hfill -(n^5+5n^4+20n^3+45n^2+70n+57)XY^2-Y^3,
\\[3mm]
F_{n,6}(X,Y)&=&X^3-(n^6+12n^4-6n^3+42n^2-30n+38)X^2Y
\\[0mm]
&&\hfill +(n^6+6n^5+27n^4+74n^3+147n^2+186n+129)XY^2-Y^3.
\\[3mm]
\end{array}\right.
$$

Let us make a few remarks about the solutions of the Thue equations 
$$
F_{n,a}(X,Y)=c\quad\mbox{ with} \quad c \in \{+1,-1\} .
$$
For all $n,a\in \Z$ with $a\neq 0$, the only solutions $(x,y)$ with $xy=0$ are given by
$$
F_{n,a}(c,0)=c \quad \mbox{and} \quad F_{n,a}(0,-c)=c,
$$
the solutions $(c,0)$ and $(0,-c) $ being dubbed {\it trivial solutions}.\\

Suppose that for all integers $n\geq 0$ and $a\geq 1$, we know the solutions of 
$F_{n,a}(X,Y)=1$. Then, because of the formulas 
$$\left\{
\begin{array}{rcc}
F_{n,a}(X,Y)&=&-F_{n,a}(-X,-Y),
\\[2mm]
 F_{-n-1,a}(X,Y)&=&F_{n,a}(-Y,-X),
\\[2mm]
\end{array}\right.
$$ 
 we can exhibit, for all integers $ n,\, a \in \Z$ with $a\geq 1$, the solutions of
$$
F_{n,a}(X,Y)=c .
$$
Moreover, because of the formulas
$$
F_{n,-a}(X,Y)=-F_{n,a}(Y,X), 
$$
 we can exhibit, for all $n, \, a\in \Z $ with $a\neq 0$, the solutions of
$$
F_{n,a}(X,Y)= c .
$$
 
The elements of the sequences $\{u_a\}_{a\geq 0} $ and $\{v_a\}_{a\geq 0} $ verify the following properties, which will prove useful in the proof of Proposition $\ref{proposition:fnc1c2=c3}$.

\medskip

\begin{lemme}\label{Lemma:Recurrence}

Let $n$ and $a$ be nonzero integers.

{\rm (i)} For $n\geq 1$ with $a\geq 1$, we have 
$ u_a > 0$, except for $ (n, a) = (1, 1)$ where $u_1=0$.

{\rm (ii)} For $n\geq 1$ with $a\geq 2$, we have $2u_a >nu_{a-1}$, except for
$(n, a) = (1, 3)$ where $2u_3=u_2$.

{\rm (iii)} For $n\geq 1$ with $a\geq 2$, we have $ v_a > u_a+v_{a-1}$, 
except for $(n, a) = (1, 2)$ where $v_2=u_2+v_1$; for $n\geq 0$ we have $v_1=u_1+v_0$; 
moreover,
for $n\geq 0$ and $a\geq 1$, we have $|u_a| \leq v_a$.
 
{\rm (iv)} For $n\geq 0$ with $a \geq 1$, we have $v_a > 2v_{a-1}$
except for $(n,a)=(0,1)$ where $(v_1,v_0)=(2,3)$ and
for $(n,a)=(0,3)$ where $(v_3,v_2)=(11,6)$ and for
 $(n,a)=(1,1)$ where $(v_1,v_0)=(3,3)$.

{\rm (v)} For $n = 0$ with $a\geq 1$, we have
$ 0<(-1)^a u_a \leq \frac12 v_a$,
except for $a=2$ where $(u_2,v_2)=(5,6)$. 
 \end{lemme}

 \begin{proof}[\indent Proof] 
 The proof of 
{\rm (i)} is easy.

{\rm (ii)} Assume $n\ge 1$. 
The inequality $2u_a \ge nu_{a-1}$ is true for $a=2,3,4$ and is strict for these values of $a$ with $(n,a)\not=(1,3)$. 
Suppose now that for some given $a \geq 4$, 
the inequality is true for $a-2$, $a-1$ and $a$. Let us prove it for $a+1$. Using the linear recurrence satisfied by the elements of $\{u_a\}_{a\geq 0}$ and the induction hypothesis, we have, 
\begin{align*}
2u_{a+1}&= 2(n-1)u_a+2(n+2)u_{a-1}+2u_{a-2}
\hfill
\\
&> (n-1)nu_{a-1}+(n+2)nu_{a-2}+nu_{a-3}
\hfill
\\
&=nu_{a-1}
\hfill.
\end{align*}

{\rm (iii)} For $n=0$ we have $v_1=u_1+v_0$. Assume $n\ge 1$. The inequality is true for $a=2,3,4$ and is strict for $a\ge 2$ with $(n,a)\not=(1,2)$. Suppose now that for some given $a \geq 4$ the inequality is true for $a-2$, $a-1$ and $a$. Let us prove it for $a+1$. 
In what follows, we will use the linear recurrence satisfied by the elements of $\{v_a\}_{a\geq 0}$,
the induction hypothesis and the results of the parts {\rm (i)} and {\rm (ii)}:
$$
\begin{array}{l}
v_{a+1}-u_{a+1}
\\[2mm]
= (n+2)v_a -(n-1)v_{a-1}-v_{a-2}-(n-1)u_{a}-(n+2) u_{a-1}-u_{a-2}
\\[2mm]
\qquad 
=v_a+ 2u_a +2v_{a-1} -v_{a-2} -(n+2)u_{a-1}-u_{a-2}+(n+1)(v_a-u_a-v_{a-1})
\\[2mm]
\qquad\qquad 
\ge v_a+ 2u_a +2v_{a-1} -v_{a-2} -(n+2)u_{a-1}-u_{a-2}
\\[2mm]
\qquad\qquad \qquad
 \ge v_a+2u_a+2(v_{a-2}+ u_{a-1})-v_{a-2}-(n+2)u_{a-1}-u_{a-2}
\\[2mm]
\qquad\qquad\qquad
 \qquad 
= v_a+(2u_a-nu_{a-1})+v_{a-2}-u_{a-2}
\\[2mm]
\qquad\qquad\qquad 
 \qquad
\qquad > v_a+v_{a-3}
\\[2mm]
\mbox{}\qquad\qquad\qquad \qquad \qquad \qquad 
> v_a. 
\\[2mm]
\end{array} 
$$
The last assertion is obvious.

{\rm (iv)} Assume that $(n,a)\neq (0,1), (1,1)$ and suppose that $n \geq 0$.
 The proof will be by induction on $a$.
The result is true for $a=1,2,3,4$.
 Assume that for some given $a \geq 4$, the result is true for $a-1$
and $a$. Let us prove it for $a + 1$: 
\begin{align}
\notag
v_{a+1}&=2v_a+nv_a-(n-1)v_{a-1}-v_{a-2}
\\
\notag
&> 2v_a+2nv_{a-1}-(n-1)v_{a-1}-v_{a-2}
\\
\notag
&
=2v_a+(n+1)v_{a-1}-v_{a-2}
\\
\notag
&
>2v_a.
\end{align}

{\rm (v)} Assume $n=0$. 
We have
$$ \left\{
\begin{array}{ccccccccc}
u_0=3,& u_1=-1,& u_2=5,& u_3=-4,& u_4=13, & u_5=-16,
\\[2mm]
v_0=3,& v_1=2,& v_2=6,& v_3=11,& v_4=26, & v_5=57.
\end{array}
\right.
$$ 
The result is true for $a = 1,2,3,4, 5$. Assume $a \geq 5$ and suppose 
that it is
true for $a-1$ and $a$. Let us prove it for $a+1$:
\begin{align*}
(-1)^{a+1}u_{a+1}&=(-1)^{a+1}(-u_a+2u_{a-1}+u_{a-2})
\\[2mm]
& \leq (-1)^a u_a +2(-1)^{a-1}u_{a-1}
\\[2mm]
 &\leq \frac{1}{2}v_a+v_{a-1} \; \leq \frac{1}{2}v_a+\frac{1}{2}v_a\le \frac{1}{2}v_{a+1}\cdotp 
\end{align*} \vskip -.7cm 
 \end{proof}

\smallskip
 
\begin{proposition}\label{proposition:fnc1c2=c3}
Suppose $n\geq 0$ and $a\geq 1$. 
The only cases where 
\linebreak
$F_{n,a}(c_1,c_2)=c$ with $c,c_1,c_2\in \{+1,-1\}$ 
are the following ones:
$$
F_{0,1}(-c,-c) = c,\qquad
F_{0,2}(c,c) = c,\qquad 
F_{n,1}(-c,c) = c \; \mbox{{\sl for any} $n\geq 0$}.
$$
\end{proposition}

 \begin{proof} [\indent Proof] 
Since $F_{n,a}(-X,-Y)=-F_{n,a}(X,Y)$, it suffices to investigate when $(1,1)$ and $(-1,1)$ are solutions of $F_{n,a}(X,Y)=c$. Suppose first that $(1,1)$ is a solution.
 The equation we consider is 
 $$
 -u_a+(-1)^a v_a=c.
 $$
 Let $a=1$; this implies $n=0$ and $c=-1$, hence $F_{0,1}(-c,-c)=c$.
Let $a=2$; this implies $n=0$ and $c=1$, hence $F_{0,2}(c,c)=c$.
Finally, in the case $a\ge 3$, we deduce from the parts (iii) and (iv) of Lemma $\ref{Lemma:Recurrence}$
 that this does not occur.\smallskip

 Suppose next that $(-1,1)$ is a solution.
 The equation we consider is 
 $$
 -u_a-(-1)^a v_a=c+2.
 $$
If $a=1$, then $c=1$ and any $n$ is admissible, hence $F_{n,1}(-c,c)=c$.
If $a=2$, this is impossible. Finally, in the case $a\ge 3$, we deduce again from the parts (iii), (iv) and (v) of Lemma $\ref{Lemma:Recurrence}$
 that this does not occur.
 \end{proof}

When $(x,y,m)\in\Z^3$ is given with $xy\not=0$ and $m>0$, the part (iv) of Theorem $\ref{Theoreme:principal}$ shows in an effective way the finiteness of the set of couples $(n,a)$ with $n\ge 0$ and $a\ge 1$ for which $0<|F_n(x,y)|\le m$, 
excluding the cases exhibited in Proposition \ref{proposition:fnc1c2=c3}. Let us give an elementary proof in the case where $y=\pm x$. 

 \begin{lemme}\label{lemme:minorationelementaire}
 Let $x\in \Z$ and $c$ in $\{+1,-1\}$. 
 For $n\ge 1$ and $a\ge 2$, we have
 $$
 |F_{n,a}(x,cx)|\ge \frac{|x|^3}{8} a n^{a-1}.
 $$
 For $n\ge 0$ and $a=1$, we have
 $$
 F_{n,1}(x,x)= - (2n+1) x^3.
 $$
 For $n=0$ and $a\ge 3$, we have
 $$
 |F_{0,a}(x,cx)|\ge |x|^3 2^{a-1}.
 $$
 \end{lemme}
 
 \begin{proof}[\indent Proof] 
 These results are trivial for $x=0$. By homogeneity we may assume $x=1$. 
 For $0<u<1$, we have
 $$
 \log (1+u)\ge \frac{u}{2}
 $$
 and for $t\ge 0$ we have
 $$
 e^t-1\ge t,
 $$
 hence
 $$
 (1+u)^a-1 =\exp \bigl( a \log (1+u)\bigr) - 1 \ge a \log (1+u)\ge \frac{au}{2}\cdotp
 $$
We use this estimate with $u= |\lambda_2|-1$. Recall that 
$$
u\ge \frac{1}{n+1}\ge \frac{1}{2n}\cdotp
$$
Hence 
$$
|1-\lambda_2^a c|\ge |\lambda_2|^a - 1=(1+u)^a-1 \ge \frac{a}{4n}\cdotp
$$
The first part of the lemma follows from the identity
$$
F_{n,a}(1,c)=(1-\lambda_0^a c)(1-\lambda_1^a c)(1-\lambda_2^a c)
$$ 
where 
$$
|(1-\lambda_0^a c)
 (1-\lambda_1^a c)|\ge
 (\lambda_0^a -1)
 (1-\lambda_1^a)\ge 
 \frac{1}{2} n^a.
$$
 
The second part of the lemma $\ref{lemme:minorationelementaire}$ follows from the explicit expression for $F_{n,1}(X,Y)$. 

 For the third part, we obtain the required inequality by using the part (iv) of Lemma $\ref{Lemma:Recurrence}$.
 \end{proof}
 \bigskip

 Let $c\in \{+1, -1\}$ and let $n,a\in \N $ with $a\geq 1$. 
We wonder whether all the solutions $(x,y)\in\Z^2$ of $F_{n,a}(x,y) =c$ are given by
	
$$ \left\{
\begin{array}{l}
\bullet \; (c,0), \; (0,c)\; \mbox{ for any } n\geq 0 \mbox{ and } a\geq 1,
\\[2mm]

\bullet \; (-c,c) \; \mbox{ for any } n\geq 0 \mbox{ and } a=1,
\\[2mm]

\bullet \; (c,c)\mbox{ for } n=0 \mbox{ and } a=2,
\\[2mm]

\bullet \; (-c,-c)\mbox{ for } n =0 \mbox{ and } a=1,
\\[2mm]

\bullet \; \mbox{the exotic solutions}
\\[3mm] 

\quad \begin{array}{c|cccccccc}
(n,a)&&(cx,cy)&&&&&
\\[1mm] \hline 
(0,1)&(-9,5)&(-1,2)&(2,-1)&(4,-9)&(5,4)
\\[1mm]
(0,2)&(-14,-9)&(-3,-1)&(-2,-1)&(1,5)&(3,2)&(13,4)&
\\[1mm]
(0,3)&(2,1)&&&&&&&
\\[1mm]
(0,5)&(-3,-1)&(19,-1)&&&&&&
\\[1mm]
(1,1)&(-3,2)&(1,-3)&(2,1)&&&&
\\[1mm]
(1,2)&(-7,-2)&(-3,-1)&(2,1)&(7,3)&&&&
\\[1mm]
(2,2)&(-7,-1)&(-2,-1)&&&&&&
\\[1mm]
(3,1)&(-7,-2)&(-2,9)&(9,-7)&&&&
\\[1mm]
(4,2)&(3,2)&&&&&&& 
\\[1mm]
\end{array}
\end{array}
\right.
$$
In the above list, the solutions associated to $a=1$ come from \cite{T} and \cite{M}.
Moreover, 
the other solutions were obtained via {\sl MAPLE} 
in the range 
$$
0\leq n\leq 10,
\qquad
2\le a \le 70,
\qquad
-1000\leq x,y \leq 1000.
$$ 

\section*{Acknowledgments} 
 We heartily thank Isao Wakabayashi for his clever remarks on a preliminary version of this paper. 
 \smallskip


\vfill

\vfill

\end{document}